\documentclass[12pt,reqno]{amsart}
\usepackage[utf8]{inputenc}
\usepackage{amsmath}
\usepackage{amssymb}
\usepackage{amsthm}
\usepackage{hyperref}
\usepackage{tikz}
\usepackage{todonotes}
\usepackage{float}
\usepackage{subcaption}
\usepackage{wrapfig}
\usepackage[margin=1in]{geometry}
\usepackage{textcomp}

\usetikzlibrary{fit}
\usetikzlibrary{patterns}
\usetikzlibrary{backgrounds}

\theoremstyle{definition}
\newtheorem{theorem}{Theorem}
\newtheorem{proposition}[theorem]{Proposition}
\newtheorem{lemma}[theorem]{Lemma}
\newtheorem{corollary}[theorem]{Corollary}
\newtheorem{definition}[theorem]{Definition}
\newtheorem{remark}[theorem]{Remark}
\newtheorem{example}[theorem]{Example}
\newtheorem{conjecture}[]{Conjecture}

\newtheorem*{theorem*}{Theorem}
\newtheorem*{lemma*}{Lemma}

\newcommand{\imm}{\mathrm{imm}}

\newcommand{\E}{\mathbb{E}}

\newcommand{\rank}{\mathrm{rank}}
\newcommand{\len}{\mathrm{len}}
\newcommand{\ET}{ET_{\le \ell}}
\newcommand{\ETsub}{ET_{\le \ell-1}}

\title[On torsion in eulerian magnitude homology of ER random graphs]{On torsion in eulerian magnitude homology of Erd\"os-R\'enyi random graphs}
\author{Giuliamaria Menara}
\date{}
\begin{document}
	\begin{sloppypar}
	\maketitle
	
	\begin{abstract}
		In this paper we investigate the regimes where an Erd\"os-R\'enyi random graph has torsion free eulerian magnitude homology groups.
		To this end, we start be introducing the eulerian Asao-Izumihara complex - a quotient CW-complex whose homology groups are isomorphic to direct summands of the graph eulerian magnitude homology group.
		We then proceed by producing a vanishing threshold for a shelling of eulerian Asao-Izumihara complex.
		This will lead to a result establishing the regimes where eulerian magnitude homology of Erd\"os-R\'enyi random graphs is torsion free. 
	\end{abstract}
	
	\section{Introduction}
	
	Magnitude, introduced by Leinster in \cite{leinster2013magnitude}, is an invariant for metric spaces that quantifies the number of effective points in the space.
	Hepworth and Willerton introduced magnitude homology for graphs as a categorification of magnitude \cite{hepworth2015categorifying}, and this concept was later extended to metric spaces and enriched categories by Leinster and Shulman \cite{leinster2021magnitude}.
	In recent years, various methods have been devised to calculate the magnitude homology groups \cite{asao2024magnitude,bottinelli2021magnitude,gu2018graph,hepworth2015categorifying,kaneta2021magnitude}.
	
	Eulerian magnitude homology is a variant recently introduced by Giusti and Menara in \cite{giusti2024eulerian} to highlight the connection between magnitude homology of simple graphs equipped with the path metric and their combinatorial structure.
	Here the authors introduce the complex of \emph{eulerian magnitude chains}, which are supported by trails without repeated vertices.
	Then they describe the strong connections between the $(k,k)$-eulerian magnitude homology groups and the graph's structure.
	Further, in the context of Erd\"os-R\'enyi random graphs they derive a vanishing threshold for the limiting expected rank of the $(k,k)$-eulerian magnitude homology in terms of the density parameter.
	
	In this paper, we will make some progress towards investigating the presence of torsion in eulerian magnitude homology.
	
	Torsion in standard magnitude homology was first studied by Kaneta and Yoshinaga in \cite{kaneta2021magnitude}, where the authors have analyzed the structure and implications of torsion in magnitude homology.
	Torsion in the magnitude homology of graphs was also studied by Sazdanovic and Summers in \cite{sazdanovic2021torsion} and by Caputi and Collari in \cite{caputi2023finite}.
	
	In the present work, as a first step towards exploring whether graphs have torsion in their eulerian magnitude homology groups, we turn our attention to Erd\"os-R\'enyi model for random graphs.
	This model is the most extensively studied and utilized model for random graphs, and it represents the maximum entropy distribution for graphs with a given expected edge proportion.
	Random complexes originating from Erd\"os-R\'enyi graphs are widely studied in stochastic topology \cite{kahle2009topology,kahle2011random,kahle2013limit}, and in studying this ``unstructured" example our intent is to create a foundation for understanding the torsion in ``structured" graphs.
	
	Adapting the construction introduced by Asao and Izumihara in \cite{asao2021geometric} to the context of eulerian magnitude homology, we able to produce for every pair of vertices $(a,b) \in G$ two simplicial complexes $\ET(a,b)$ and $\ETsub(a,b)$ such that the homology of the quotient $\ET(a,b) / \ETsub(a,b)$ is isomorphic to a direct summand of the eulerian magnitude homology  $EMH_{\ast,\ell}(G)$ up to degree shift.
	Therefore, producing a shellability result of the complexes $\ET(a,b)$ and $\ETsub(a,b)$ will in turn determine a torsion-free result for $EMH_{\ast,\ell}(G)$.
	In Theorem \ref{thm:ET(a,b) shellable} we achieve such shellability result for $\ET(a,b)$ in terms of the density parameter.
	Further, in Corollary \ref{cor:if nonvanishing then torsion free} we link the torsion-free result for eulerian magnitude homology groups stated in Theorem \ref{thm:torsion free} with the vanishing threshold produced in \cite[Thm. 4.4]{giusti2024eulerian}, determining sufficient conditions under which if eulerian magnitude homology is non-vanishing, then it is also torsion-free.

	\subsection{Outline}
	The paper is organized as follows. 
	We start by recalling in Section \ref{sec:Background} some general background about graphs, eulerian magnitude homology and shellability.
	In Section \ref{sec:A-I for EMH} we introduce the \emph{eulerian} Asao-Izumihara complex.
	We then investigate in Section \ref{sec:homotopy type EMC} the probability regimes in which the eulerian Asao-Izumihara complex is shellable, and we conclude by producing a vanishing threshold for torsion in eulerian magnitude homology groups.	
	Finally, in Section \ref{sec:future directions} we propose extensions of the current work and identify open questions that could deepen the understanding of the topic.

	\section{Background}
	\label{sec:Background}
	
	We begin by recalling relevant definitions and results.
	We assume readers are familiar with the general theory of simplicial homology (for a thorough exposition see~\cite{hatcher2005algebraic}).
	Throughout the paper we adopt the notation $[m] = \{1, \dots, m\}$ and $[m]_0 = \{0, \dots, m\}$ for common indexing sets.
	
	\subsection{Graph terminology and notation}
	\label{subsec:graph definitions}
	
	An undirected graph is a pair $G=(V,E)$ where $V$ is a set of vertices and $E$ is a set of edges (unordered pairs of vertices).
	A \emph{walk} in such a graph $G$ is an ordered sequence of vertices $x_0,x_1,\ldots,x_k\in V$ such that for every index $i \in [k]_0$ there is an edge $\{x_i,x_{i+1}\}\in E$.
	A \emph{path} is a walk with no repeated vertices.
	For the purposes of introducing eulerian magnitude homology we assume that all graphs are simple, i.e. they have no self-loops and no multiedges~\cite{leinster2019magnitude}.
	One can interpret the set of vertices of a graph as an extended metric space (i.e. a metric space with infinity allowed as a distance) by taking the \emph{path metric} $d(u,v)$ to be equal to the length of a shortest path in $G$ from $u$ to $v$, if such a path exists, and taking $d(u,v) = \infty$ if $u$ and $v$ lie in different components of $G$.
	
	\begin{definition}
		\label{def:ktrail}
		Let $G = (V,E)$ be a graph, and $k$ a non-negative integer. A \emph{$k$-trail} $\bar{x}$ in $G$ is a $(k+1)$-tuple $(x_0,\dots,x_k) \in V^{k+1}$ of vertices for which $x_i \neq x_{i+1}$ and $d(x_i,x_{i+1})<\infty$ for every $i \in [k-1]_0$.
		The \emph{length} of a $k$-trail $(x_0,\dots,x_k)$ in $G$ is defined as the minimum length of a walk that visits $x_0,x_1,\ldots,x_k$ in this order:
		\[
		\len(x_0,\dots,x_k) = d(x_0,x_1)+\cdots + d(x_{k-1},x_k).
		\]
		We call the vertices $x_0, \dots x_{k}$ the \emph{landmarks}, $x_0$ the \emph{starting point}, and $x_k$ the \emph{ending point} of the $k$-trail.
	\end{definition}

	\subsection{Eulerian magnitude homology}
	\label{subsec:EMH definitions}
	
	\sloppy
	The \emph{magnitude homology} of a graph $G$, $MH_{k,\ell}(G)$, was first introduced by Hepworth and Willerton in~\cite{hepworth2015categorifying}, and the \emph{eulerian} magnitude homology of a graph $EMH_{k,\ell}(G)$ is a variant of it with a stronger connection to the subgraph structures of $G$.
	Specifically, while the building blocks of standard magnitude homology are tuples of vertices $(x_0,\dots,x_k)$ where we ask that \emph{consecutive} vertices are different, eulerian magnitude homology is defined starting from tuples of vertices $(x_0,\dots,x_k)$ where we ask that \emph{all} landmarks are different.
	
	Eulerian magnitude homology was recently introduced by Giusti and Menara in~\cite{giusti2024eulerian} and we recall here the construction.
	
	\begin{definition}(Eulerian magnitude chain)
		Let $G=(V,E)$ be a graph.
		We define the $(k,\ell)$-eulerian magnitude chain $EMC_{k,\ell}(G)$ to be the free abelian group generated by trails $(x_0,\dots,x_k) \in V^{k+1}$ such that $x_i \neq x_j$ for every $0\leq i,j \leq k$ and $\len(x_0,\dots,x_k)=\ell$.
	\end{definition}
	
	It is straightforward to demonstrate that the eulerian magnitude chain is trivial when the length of the path is too short to support the necessary landmarks.
	
	\begin{lemma}[{c.f. \cite[Proposition 10]{hepworth2015categorifying}}]
		\label{lem:LowerTriangular}
		Let $G$ be a graph, and $k > \ell$ non-negative integers. Then $EMC_{k,\ell}(G) \cong 0.$
	\end{lemma}
	
	\begin{proof}
		Suppose $EMC_{k,\ell}(G)\neq 0.$
		Then, there must exist a $k$-trail $(x_0,\dots,x_k)$ in $G$ so that $\len(x_0,\dots,x_k)=d(x_0,x_1)+\cdots+d(x_{k-1},x_k) = \ell$.
		However, as all vertices in the $k$-trail must be distinct, $d(x_i,x_{i+1}) \geq 1$ for $i \in [k-1]_0$, so $k$ can be at most $\ell$.
	\end{proof}
	
	\begin{definition}(Differential)
		\label{def:differential}
		Denote by $(x_0,\dots,\hat{x_i},\dots,x_k)$ the $k$-tuple obtained by removing the $i$-th vertex from the $(k+1)$-tuple $(x_0,\dots,x_k)$.  We define the \emph{differential}
		\[
		\partial_{k,\ell}: EMC_{k,\ell}(G) \to EMC_{k-1,\ell}(G)
		\]
		to be the signed sum $\partial_{k,\ell}= \sum_{i\in [k-1]}(-1)^{i}\partial_{k,\ell}^i$ of chains corresponding to omitting landmarks without shortening the walk or changing its starting or ending points,
		\[
		\partial_{k,\ell}^i(x_0,\dots,x_k) = \begin{cases}
			(x_0,\dots,\hat{x_i},\dots,x_k) , &\text{ if } \len(x_0,\dots,\hat{x_i},\dots,x_k) = \ell, \\
			0, &\text{ otherwise.}\\
		\end{cases}
		\]
	\end{definition}
	
	For a non-negative integer $\ell$, we obtain the \emph{eulerian magnitude chain complex}, $EMC_{*,\ell}(G),$ given by the following sequence of free abelian groups and differentials.
	
	\begin{definition}(Eulerian magnitude chain complex)
		We indicate as $EMC_{*,\ell}(G)$ the following sequence of free abelian groups connected by differentials
		\[
		\cdots \rightarrow EMC_{k+1,\ell}(G) \xrightarrow{\partial_{k+1,\ell}} EMC_{k,\ell}(G) \xrightarrow{\partial_{k,\ell}} EMC_{k-1,\ell}(G) \to \cdots
		\]
	\end{definition}
	
	The differential map used here is the one induced by standard magnitude, and it is shown in \cite[Lemma 11]{hepworth2015categorifying} that the composition $\partial_{k,\ell} \circ \partial_{k+1,\ell}$ vanishes, justifying the name ``differential" and allowing the definition the corresponding bigraded homology groups of a graph.
	
	\begin{definition}(Eulerian magnitude homology)
		\label{def_EMH}
		The $(k,\ell)$-eulerian magnitude homology group of a graph $G$ is defined by
		\[
		EMH_{k,l}(G) = H_k(EMC_{*,l}(G)) = \frac{\ker(\partial_{k,\ell})}{\imm(\partial_{k+1,\ell})}.
		\]
	\end{definition}
	
	\begin{figure}
		\begin{minipage}{0.3\textwidth}
			
			\begin{tikzpicture}[node distance={14mm}, thick, main/.style = {draw, circle}]
				\node[main] (0) {$0$}; 
				\node[main] (1) [right of=0] {$1$};  
				\node[main] (2) [above of=1] {$2$};
				\node[main] (3) [right of=1] {$3$};
				\node[main] (4) [above of=3] {$4$};
				\draw (0) -- (1);
				\draw (1) -- (2);
				\draw (0) -- (2);
				\draw (2) -- (3);
				\draw (2) -- (4);
				\draw (3) -- (4);
			\end{tikzpicture} 
			\caption{Graph $G$}
			\label{fig:toyexampleEMH}
			
		\end{minipage}
		\hfill
		\begin{minipage}{0.68\textwidth}
			\begin{example}
				\label{ex:toyexampleEMH}
				We will compute $EMH_{2,2}(G)$ for the graph $G$ in Figure \ref{fig:toyexampleEMH}.
				$EMC_{2,2}(G)$ is generated by the $2$-paths in $G$ of length $2$. There are twenty such paths, consisting of all possible walks of length two in the graph visiting different landmarks: (0,1,2), (0,2,1), (0,2,3), (0,2,4), (1,0,2), (1,2,0), (1,2,3), (1,2,4), (2,0,1), (2,1,0), (2,3,4), (2,4,3), (3,2,0), (3,2,1), (3,2,4), (3,4,2), (4,2,0), (4,2,1), (4,2,3), (4,3,2).
				Similarly, $EMC_{1,2}(G)$ is generated by the eight $1$-paths in $G$ of length $2$:(0,3), (0,4), (1,3), (1,4), (3,0), (3,1), (4,0), (4,1).
				Because $\partial_{2,2}$ only omits the center vertex, it is easy to check that the kernel is generated by the $12$ elements visiting the triangle with vertices 0,1,2 and the one with vertices 2,3,4.  
				Also, by Lemma \ref{lem:LowerTriangular}, $EMC_{3,2}(G)$ is the trivial group, and thus the image of $\partial_{3,2}$ is $\langle 0 \rangle$. Thus, $\rank(EMH_{2,2}(G))=12$, generated by those walks between vertices 0,1,2 and 2,3,4.
			\end{example}
		\end{minipage}
	\end{figure}

	Notice that by construction we have the following proposition.
	
	\begin{proposition}
		For $\ell\ge 0$, the following direct sum decomposition holds:
		\[ 
		EMC_{\ast, \ell}(G) =\bigoplus_{a, b\in V(G)} EMC_{\ast, \ell}(a, b),
		\]
		where $EMC_{\ast, \ell}(a, b)$ is the subcomplex of $EMC_{\ast, \ell}(G)$ generated by trails which start at $a$ and end at $b$.
	\end{proposition}

	\subsection{Shellable simplicial complexes}
	\label{subsec:shellable simplcial complexes}
	
	We recall the definition of shellable simplicial complex.
	
	\begin{definition}[{\cite[Definition 2.1]{bjorner1996shellable}}]
		\label{def:shellable}
		If $X$ is a finite simplicial complex, then a \emph{shelling} of $X$ is an ordering $F_1,\dots,F_t$ of the facets (maximal faces) of $X$ such that $F_k \cap \bigcup_{i=1}^{k-1}F_i$ is a non-empty union of facets of $F_k$ for $k \geq 2$.
		If $X$ has a shelling, we say it is \emph{shellable}.
	\end{definition}
	
	In other words, we ask that the last simplex $F_k$ meets the previous simplices along some union $B_k$ of top-dimensional simplices of the boundary of $F_k$, so that $X$ can be built stepwise by introducing the
	facets one at a time and attaching each new facet $F_k$ to the complex previously built in the nicest possible fashion.
	
	Suppose $X$ is a non-pure simplicial complex.
	In this case the first facet of a shelling is always of maximal dimension.
	In fact, if $X$ is shellable there is always a shelling in which the facets appear in order of decreasing dimension.
	
	\begin{lemma}[{\cite[Rearrangement lemma, 2.6]{bjorner1996shellable}}]
		\label{lem:structure nonpure shelling}
		Let $F_1, F_2, \dots ,F_t$ be a shelling of $X$.
		Let $F_{i_1}, F_{i_2}, \dots ,F_{i_t}$ be the rearrangement obtained by taking first all
		facets of dimension $d = \dim X$ in the induced order, then all facets of dimension $d-1$ in the induced order, and continuing this way in order of decreasing dimension.
		Then this rearrangement is also a shelling.
	\end{lemma}
	
	\begin{theorem}[{\cite[Theorem 2.9]{bjorner1996shellable}}]
		\label{thm:structure nonpure shelling}
		Let $X$ be a simplicial complex, and let $0 \leq r \leq s \leq \dim X$.
		Define $X^{(r,s)}=\{\sigma \in X \text{ such that } \dim \sigma \leq s \text{ and } \sigma \in F \text{ for some facet $F$ with } \dim F \geq r  \}$.
		If $X$ is shellable, then so is $X^{(r,s)}$ for all $r \leq s$.
	\end{theorem}
	
	Lemma \ref{lem:structure nonpure shelling} and Theorem \ref{thm:structure nonpure shelling} can be interpreted as providing a kind of ``structure theorem'', describing how a general shellable complex $X$ is put together from pure
	shellable complexes.
	First there is the pure shellable complex $X^1=X^{(d,d)}$ generated by all facets of maximal size.
	Then $X^1$'s $(d-1)$-skeleton, which is also shellable, is extended by shelling steps in dimension $d-1$ to obtain $X^2=X^{(d-1,d)}$.
	Then $X^2$'s $(d-2)$-skeleton is extended by shelling steps in dimension $(d-2)$ to obtain $X^{(d-2,d)}$, and so on until all of $X=X^{(0,d)}$ has been constructed.
	
	A shellable simplicial complex enjoys several strong properties of a combinatorial, topological and algebraic nature.	
	Let it suffice here to mention that it is homotopy equivalent to a wedge sum of spheres, one for each spanning simplex of corresponding dimension~\cite{farmer1978cellular}.

	\section{Eulerian Asao-Izumihara complex}
	\label{sec:A-I for EMH}
	
	We introduced in this section the \emph{eulerian} Asao-Izumihara complex.
	
	Recall that the Asao-Izumihara complex is a CW complex which is obtained as the quotient of a simplicial complex $K_{\ell} (a,b)$ divided by a subcomplex $K'_{\ell} (a,b)$, and was proposed in~\cite{asao2021geometric} as a geometric approach to compute magnitude homology of general graphs.
	Here we adapt this construction to the context of eulerian magnitude homology, providing a way of replacing the computation of the eulerian magnitude homology $EMH_{k,\ell}(G)$ by that of simplicial homology.  
	\\\\
	Let us start by recalling the Asao-Izumihara complex.
	Let $G=(V,E)$ be a connected graph and fix $k \geq 1$.
	For any $a,b \in V$ the set of walks with length $\ell$ which start with $a$ and end with $b$ is denoted by
	\[
	W_{\ell}(a,b):=\{\bar{x}=(x_0,\dots,x_k) \text{ walk in $G$ } | x_0=a,x_k=b, \len(\bar{x})=\ell \}.
	\]

	\begin{definition}[c.f.{\cite[Def. 4.1]{asao2021geometric}}]
		Let G be a graph, and $a, b \in V$, $\ell \geq 3$.
		\begin{alignat*}{3}
			&K_{\ell}(a,b) := \{ &&\emptyset \neq ((x_{i_1},i_1),\dots,(x_{i_k},i_k)) \subset V \times \{1,\dots,\ell -1\} \\
			&	&& | (a, x_{i_1},\dots,x_{i_k},b) \prec\exists (a, x_1,\dots, x_{\ell -1}, b) \in W_{\ell}(a,b) \} \\
			&K'_{\ell}(a,b) := \{&& ((x_{i_1},i_1),\dots,(x_{i_k},i_k)) \in K_{\ell}(a,b) | \len(a, x_{i_1},\dots,x_{i_k},b) \leq \ell -1 \}.
		\end{alignat*}
	\end{definition}
	
	\begin{remark}
		Following \cite{asao2021geometric}, we will denote $((x_{i_1},i_1),\dots,(x_{i_k},i_k))$ by $(x_{i_1},\dots,x_{i_k})$ when there is no confusion.
		
		It can also be easily seen that $K_{\ell}(a,b)$ is a simplicial complex and $K'_{\ell}(a,b)$ is a subcomplex.
	\end{remark}
	
	\begin{theorem}[c.f.{\cite[Thm. 4.3]{asao2021geometric}}]
		Let $\ell \geq 3$ and $\ast \geq 0$.
		Then, the isomorphism
		\[
		(C_{\ast}(K_{\ell}(a,b),K'_{\ell}(a,b)),-\partial) \cong
		(MC_{\ast +2,\ell}(a,b),\partial)
		\]
		of chain complexes holds.
	\end{theorem}
	
	\begin{corollary}[c.f.{\cite[Cor. 4.4]{asao2021geometric}}]
		Let $\ell \geq 3$. 
		\begin{itemize}
			\item If $k \geq 3$, $MH_{k,\ell}(a,b) \cong H_{k-2}(K_{\ell}(a,b),K'_{\ell}(a,b))$.
			\item If $k = 2$, we also have
			\[
			MH_{2, \ell}(a, b) \cong \begin{cases}H_{0}(K_{\ell}(a, b), K'_{\ell}(a, b)) & \text{ if } d(a, b) < \ell, \\ \tilde{H}_{0}(K_{\ell}(a, b))&\text{ if } d(a, b) = \ell,\end{cases}
			\]
			where $\tilde{H}_{\ast}$ denotes the reduced homology group. 
		\end{itemize}
	\end{corollary}
	
	\begin{remark}
		Notice while both $K_{\ell -1}(a,b)$ and $K'_{\ell}(a,b)$ are subcomplexes of $K_{\ell}(a,b)$, in general $K_{\ell -1}(a,b) \subsetneq K'_{\ell}(a,b)$.
		Indeed, say $v$ and $u$ are two adjacent vertices, then the tuple $(v,u,u)$ is an element of both $K_3(v,u)$ and $K'_3(v,u)$ because it is a subtuple of $(v,u,v,u)$, but it cannot be in $K_2(v,u)$.
		This type of example with consecutively repeated vertices is the only one that can be constructed to show that $K_{\ell-1}(a,b)$ is a proper subset of $K'_\ell(a,b)$, and in the context of eulerian magnitude homology it cannot arise because the tuples have all different vertices.
		Therefore when introducing the eulerian Asao-Izumihara complex it will possible to only rely on the (eulerian versions of the) complexes $K_{\ell}(a,b)$ and $K_{\ell -1}(a,b)$.
	\end{remark}
	
	\begin{definition}
		\label{def:ET(a,b)}
		Let $\ET(a, b)$ be the set of eulerian trails from $a$ to $b$ with length smaller than $\ell$. 
		That is, the set of all trails $(x_1, \dots, x_t)\in V^{t+1}$ such that $x_i \neq x_j$ for every $i, j \in \{1,\dots,t\}$ and 
		\[
		\len(a,x_1, \dots, x_t,b) \le \ell.
		\]
	\end{definition}
	
	The set $\ET(a, b)$ is clearly a simplicial complex, and the complex $\ETsub(a, b)$ is a subcomplex of $\ET(a, b)$, see Figure \ref{fig:eulerian AI} for an illustration.
	
	\begin{example}
		\sloppy
		Consider the same graph $G$ as in example \ref{ex:toyexampleEMH}.
		Suppose we choose $(a,b)=(0,4)$ and $\ell =4$.
		Then we have $ET_4(0,4)=\{(1,2,3),(1,2),(1,3),(2,3),(1),(2),(3) \}$ and $ET_3(0,4)=\{(1,2),(2,3),(1),(2),(3) \}$.
		
		\begin{figure}[h]
			\centering
			\begin{tikzpicture}[node distance={20mm}, thick, main/.style = {draw, circle}]
				\node[main,red,fill=white] (1) {$1$};   
				\node[main,red,fill=white] (3) [above right of=1] {$3$};
				\node[main,red,fill=white] (2) [below right of=3] {$2$};
				\draw[red] (1) -- (2);
				\draw (1) -- (3);
				\draw[red] (2) -- (3);
				
				\begin{scope}[on background layer]
					\fill[gray!30] (1.center)--(3.center)--(2.center);
				\end{scope}
				
			\end{tikzpicture} 
			
			\caption{The geometric realization of $ET_{\leq 4}(0,4)$ and $ET_{\leq 3}(0,4)$: $ET_{\leq 4}(0,4)$ is the full triangle, while $ET_{\leq 3}(0,4)$ is the subcomplex represented in red.}
			\label{fig:eulerian AI}
		\end{figure}
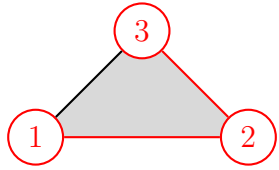
	\end{example}
	
	The following two results can be shown proceeding similarly to the proofs of \cite[Thm. 4.3 and Cor. 4.4]{asao2021geometric}.
	
	\begin{theorem}
		\label{thm:asao-izu isomorphism thm}
		
		Let $a, b$ be vertices of a graph $G$, and fix an integer  $\ell\ge 3$. 
		Then we can construct a pair of simplicial complexes $(\ET(a, b),\ETsub(a, b))$ which satisfies
		\[
		C_{\ast -2}(\ET(a, b), \ETsub(a, b)) \cong EMC_{\ast,\ell}(a, b).
		\]
	\end{theorem}

	\begin{corollary}
		\label{cor:homology congruence}
		Let $\ell \geq 3$. Then
		\[
		EMH_{k, \ell}(a, b)\cong H_{k-2}(\ET(a, b), \ETsub(a, b))
		\]
		Moreover, for $k = 2$, we also have
		\[
		EMH_{2, \ell}(a, b) \cong \begin{cases}H_{0}(\ET(a, b), \ETsub(a, b)) & \text{ if } d(a, b) < \ell, \\ \tilde{H}_{0}(\ET(a, b))&\text{ if } d(a, b) = \ell,\end{cases}
		\]
		where $\tilde{H}_{\ast}$ denotes the reduced homology group.
	\end{corollary}
	
	\section{Torsion in EMH of Erd\H{o}s-R\'{e}nyi random graphs}
	\label{sec:homotopy type EMC}
	
	In this section we investigate the regimes where the eulerian magnitude homology of Erd\H{o}s-R\'{e}nyi random graphs is torision free.
	
	Recall that the \emph{Erd\H{o}s-R\'{e}nyi (ER) model} for random graphs, denoted as $G(n,p)$ and first introduced in~\cite{erdHos1960evolution}, is one of the most extensively studied and utilized models for random graphs.
	This model represents the maximum entropy distribution for graphs with a given expected edge proportion, making it a valuable null model across a wide array of scientific and engineering fields. Consequently, the clique complexes of ER graphs have garnered significant interest within the stochastic topology community~\cite{kahle2009topology,kahle2011random,kahle2013limit}.
	
	\begin{definition}
		The \emph{Erd\H{o}s-R\'{e}nyi (ER) model} 
		$G(n, p) = (\Omega, P)$ is the probability space where $\Omega$ is the discrete space of all graphs on $n$ vertices, and
		$P$ is the probability measure that assigns to each graph $G \in \Omega$ with $m$ edges probability
		\[
		P(G)=p^m(1-p)^{{n \choose 2}-m}.
		\]
	\end{definition}
	
	We can sample an ER graph $G \sim G(n, p)$ on $n$ vertices with parameter $p\in [0,1]$ by determining whether each of the $n \choose 2$ potential edges is present via independent draws from a Bernoulli distribution with probability $p$.
	In order to study the limiting behavior of these models as $n \to \infty$, it is often useful to change variables so that $p$ is a function of $n$.
	Here we will take $p=n^{-\alpha}$, $\alpha \in [0,\infty)$, as in~\cite{giusti2024eulerian}.
	
	We will first prove in Section~\ref{subsec:homotopy type AI complex} that, under certain assumptions, the complex $\ET(a,b)$ is shellable for every choice for $\ell \geq 3$.
	This will imply that $H_{\ast}(\ET(a,b),\ETsub(a,b))$ is torsion free, and by Corollary \ref{cor:homology congruence} that $EMH_{\ast+2,\ell}(G)$ is torsion free.
	
	\subsection{Homotopy type of the eulerian Asao-Izumihara complex}
	\label{subsec:homotopy type AI complex}
	
	\sloppy
	Recall from Section~\ref{sec:A-I for EMH} that the eulerian Asao-Izumihara chain complex is the relative complex $C_{\ast}(\ET(a,b) , \ETsub(a,b))$, where $\ET(a,b)$ is the set of eulerian tuples $(x_0\dots,x_k)$ such that $\len(a,x_0,\dots,x_k,b) \leq \ell$, and $\ETsub(a,b)$ is defined similarly.
	Fix and integer $\ell \geq 3$. 
	
	\begin{theorem}
		\label{thm:ET(a,b) shellable}
		Let $G(n,n^{-\alpha})$ be an ER graph.
		Suppose the facets $f_1,\dots,f_{t-1},f_t$ of $\ET(a,b)$ are ordered in decreasing dimension. 
		Then as $n\to \infty$ $ET_{\leq \ell}(a,b)$ is shellable asymptotically almost surely when	
		\begin{itemize}
			\item $0< \alpha < \prod_{i=1}^{t-1} \frac{\dim f_i + \dim f_{i+1}}{\ell +2\dim f_{i+1} -2}$, if $\dim f_1  < \frac{\ell -2}{2}$,
			\item  $0<\alpha< \prod_{i=1}^{k-1} \frac{\dim f_i +3}{\ell +4} \prod_{i=k}^{t-1} \frac{\dim f_i + \dim f_{i+1}}{\ell +2\dim f_{i+1} -2}$, if $\dim f_i \geq \frac{\ell -2}{2}$ for $1 \leq i \leq k-1$ and $\dim f_i < \frac{\ell -2}{2}$ for $i \geq k$.
		\end{itemize}
		
	\end{theorem} 
	
	\begin{proof}
		Consider the facets $f_1,\dots,f_t$ of $\ET(a,b)$.
		Suppose they are ordered in decreasing dimension and say $\dim f_1=d$.
		There are some cases we need to consider.
		\begin{enumerate}
			\item If there is a single facet $f_1$, then $\ET(a,b)$ is homotopic to a sphere $S^{d-1}$ with $d=\dim f_1$ and we are done.
			\item Say there are two different maximal facets, $f_1$ and $f_2$ and suppose they have the same dimension $d$.\\
			If $f_1$ and $f_2$ differ in one vertex, then they intersect in a $(d-1)$-face, and thus $\{f_1,f_2\}$ is a shelling.\\		 	
			If $f_1$ and $f_2$ differ in two vertices $u,v$, then we need to distiguish the situations when $u$ and $v$ are adjacent and when they are not.
			\begin{enumerate}
				\item If $u$ and $v$ are not adjacent, then we will have $f_1=(a,\dots,u,\dots,v,\dots,b)$ and $f_2=(a,\dots,u',\dots,v',\dots,b)$, and by construction there exists a third facet $f_3=(a,\dots,u',\dots,v,\dots,b)$ such that $\{f_1,f_3,f_2\}$ is a shelling, see Figure~\ref{fig:shelling two non adjancet vertices}.
				
				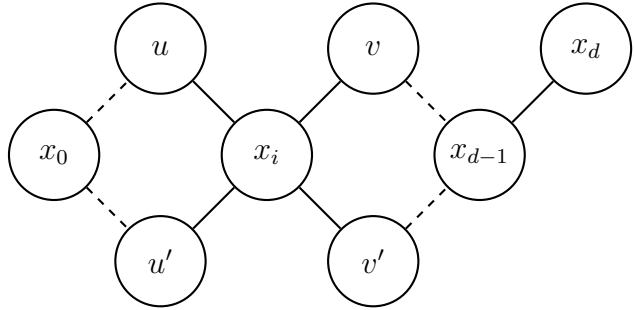
\begin{figure}[h]
					\centering
					\begin{tikzpicture}[node distance={20mm}, thick, main/.style = {draw, circle}, minimum size=1.2cm]
						
						\node[main] (0) {$x_0$}; 
						\node[main] (1) [above right of=0] {$u$}; 
						\node[main] (1')[below right of=0] {$u'$};
						\node[main] (2) [below right of=1] {$x_i$};
						\node[main] (3) [above right of=2] {$v$};
						\node[main] (3') [below right of=2] {$v'$};
						\node[main] (4) [below right of=3] {$x_{d-1}$};
						\node[main] (5) [above right of=4] {$x_d$};
						
						\draw[dashed] (0) -- (1);
						\draw[dashed] (0) -- (1');
						
						\draw (1) -- (2);
						\draw (1') -- (2);
						
						\draw (2) -- (3);
						\draw (2) -- (3');
						
						\draw[dashed] (3) -- (4);
						\draw[dashed] (3') -- (4);
						
						\draw (4) -- (5);
						
					\end{tikzpicture} 
					\caption{In this example $f_1=(x_0\dots,u,x_i,v,\dots,x_d)$ and $f_2=(x_0\dots,u',x_i,v',\dots,x_d)$. We can define $f_3=(x_0\dots,u',x_i,v,\dots,x_d)$, so that $\{f_1,f_3,f_2\}$ is a shelling.}
					\label{fig:shelling two non adjancet vertices}
				\end{figure}

				\item If  $u$ and $v$ are adjacent, then in order to construct a facet $f_3$ intersecting $f_1$ in a $(d-1)$-face we need either the edge $(u,v')$ or the edge $(u',v)$ to be present (see Figure~\ref{fig:shelling two adjancet vertices}), and this happens with probability $p=n^{-\alpha}$.
			\end{enumerate}
			
			\begin{figure}[h]
				\centering
				\begin{tikzpicture}[node distance={20mm}, thick, main/.style = {draw, circle}, minimum size=1.2cm]
					
					\node[main] (0) {$x_0$}; 
					\node[main] (1) [above right of=0] {$u$}; 
					\node[main] (1')[below right of=0] {$u'$};
					\node[main] (3) [right of=1] {$v$};
					\node[main] (3') [right of=1'] {$v'$};
					\node[main] (4) [below right of=3] {$x_{d-1}$};
					\node[main] (5) [right of=4] {$x_d$};
					
					\draw[dashed] (0) -- (1);
					\draw[dashed] (0) -- (1');
					
					\draw (1) -- (3);
					\draw (1') -- (3');
					
					\draw[dashed, red] (1) -- (3');
					\draw[dashed, red] (1') -- (3);
					
					\draw[dashed] (3) -- (4);
					\draw[dashed] (3') -- (4);
					
					\draw (4) -- (5);
					
				\end{tikzpicture} 
				\caption{In this example $f_1=(x_0\dots,u,v,\dots,x_d)$ and $f_2=(x_0\dots,u',v',\dots,x_d)$. In case one of the two dotted red edges $(u,v')$ and $(u',v)$ is present we can define $f_3=(x_0\dots,u',v,\dots,x_d)$, or $f_4=(x_0\dots,u,v',\dots,x_d)$, so that $\{f_1,f_3,f_2\}$ or $\{f_1,f_4,f_2\}$ is a shelling.}
				\label{fig:shelling two adjancet vertices}
			\end{figure}
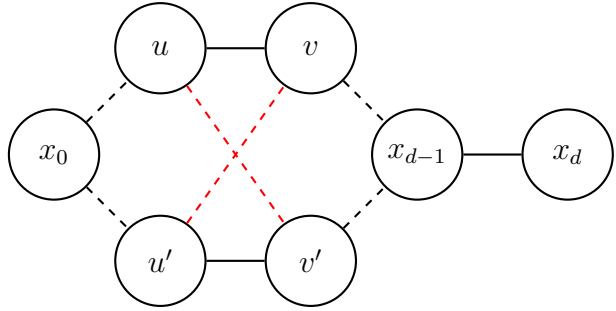

			\noindent Now say $f_1$ and $f_2$ differ in $m$ vertices and, indicating the facets $f_1$ and $f_2$ only by the vertices they differ in,  write $f_1=(u_1,u_2,\dots,u_m)$ and $f_2=(u_1',u_2',\dots,u_m')$.
			Define a partition $A_i$ with $\bigcup_{i}A_i=\{u_1,\dots,u_m\}$ such that two vertices $u_{\alpha}^{i}, u_{\beta}^{i}$ belong to the same set $A_i$ if and only if they are adjacent in $G$, see Figure~\ref{fig:shelling any adjancet vertices}.
			Call $A'_i$ the corresponding partition for the vertices $(u_1',u_2',\dots,u_m')$.
			Notice that $|A_i|=|A'_i|$ for every $i$. 
			Indeed, suppose by contradiction this is not true.
			Then, because $f_1$ and $f_2$ have the same dimension, there exists $i_1,i_2$ such that $|A_{i_1}| > |A_{i_1}'|$ and $|A_{i_2}| < |A_{i_2}'|$.
			But then it is possible to construct a $f_3$ visiting vertices from $A_{i_1}$ and $A_{i_2}'$ thus having $\dim f_3 > \dim f_1,\dim f_2$, contradicting the fact that $f_1$ and $f_2$ are maximal facets.
			
			Then in this case we need for every set of adjacent vertices $A_i$ and $A'_i$ a number $|A_i|-1$ of edges $(u_{\alpha}^i, u_{\beta}^{\prime, i})$, $\alpha \neq \beta$, in order to create a shelling.
			Indeed, we need to be able to construct a sequence of facets $f'_{1},\dots,f'_m$ by changing one vertex each time so that the intersection between the $j$-th facet and the preceding $(j-1)$ facets is a $(d -1)$- dimensional simplex, see Figure~\ref{fig:shelling any adjancet vertices}.  
			Given the fact that we also require for every set $A_i$ a number $|A_i|+1$ of edges to connect the vertices in $A_i$, we obtain that the probability of all the required edges existing is
			\[
			p^{\ell + \sum_{i}(|A_i|+1) + \sum_{i}(|A_i|-1)} =  p^{\ell +2m}.
			\] 
			With $p=n^{-\alpha}$, $\alpha \in [1/2,\infty)$, we get  
			\begin{align*}
				&\sum_{m=2}^{d-1} \binom{n}{d+1+m} n^{-\alpha (\ell + 2m)} \leq \\
				&(d-2) \binom{n}{d+3} n^{-\alpha (\ell + 4)} \sim \\
				& (d-2) \frac{n^{d+3}}{(d+3)!} n^{-\alpha (\ell + 4)}
				\xrightarrow{n\to \infty} \begin{cases}
					0, \text{ if } \alpha > \frac{d+3}{\ell +4} \\
					\infty, \text{ if } \alpha < \frac{d+3}{\ell +4}.
				\end{cases}
			\end{align*}
			
			Notice that we assumed $\alpha \in [1/2,\infty)$ and $\frac{d+3}{\ell +4} \geq \frac{1}{2}$ only when $d \geq \frac{\ell -2}{2}$.

			With $p=n^{-\alpha}$, $\alpha \in [0,1/2)$, we get  
			\begin{align*}
				&\sum_{m=2}^{d-1} \binom{n}{d+1+m} n^{-\alpha (\ell + 2m)} \leq \\
				&(d-2) \binom{n}{2d} n^{-\alpha (\ell + 2d -2)} \sim \\
				& (d-2) \frac{n^{2d}}{(2d)!} n^{-\alpha (\ell + 2d -2)}
				\xrightarrow{n\to \infty} \begin{cases}
					0, \text{ if } \alpha > \frac{2d}{\ell + 2d -2} \\
					\infty, \text{ if } 0<\alpha < \frac{2d}{\ell + 2d -2}.
				\end{cases}
			\end{align*}
			
			Since it holds also in this case that $\frac{2d}{\ell + 2d -2} \geq \frac{1}{2}$ if and only if $d \geq  \frac{\ell -2}{2}$, we can conclude that we can construct a shelling when
			\[
			\begin{cases}
				0< \alpha < \frac{d+3}{\ell +4}, \text{ if } d\geq \frac{\ell -2}{2} \\
				0<\alpha< \frac{2d}{\ell +2d -2}, \text{ if } d< \frac{\ell -2}{2}.
			\end{cases}
			\]
			
			\begin{figure}[h]
				\centering
				\begin{tikzpicture}[node distance={20mm}, thick, main/.style = {draw, circle}, minimum size=1cm]
					
					\node[main] (0) {$x_0$}; 
					\node[main] (u1) [above right of=0] {$u_1$}; 
					\node[main] (u1')[below right of=0] {$u_1'$};
					\node[main] (u2) [right of=u1] {$u_2$};
					\node[main] (u2') [right of=u1'] {$u_2'$};
					\node[main] (i) [below right of=u2] {$x_i$};
					\node[main] (u3) [above right of=i] {$u_3$}; 
					\node[main] (u3')[below right of=i] {$u_3'$};
					\node[main] (u4) [right of=u3] {$u_4$};
					\node[main] (u4') [right of=u3'] {$u_4'$};
					\node[main] (u5) [right of=u4] {$u_5$};
					\node[main] (u5') [right of=u4'] {$u_5'$};
					\node[main] (d) [below right of=u5] {$x_d$};
					
					\draw[dashed] (0) -- (u1);
					\draw[dashed] (0) -- (u1');
					
					\draw (u1) -- (u2);
					\draw (u1') -- (u2');
					
					\draw (u2) -- (i);
					\draw (u2') -- (i);
					
					\draw[dashed] (i) -- (u3);
					\draw[dashed] (i) -- (u3');
					
					\draw (u3) -- (u4);
					\draw (u3') -- (u4');
					\draw (u4) -- (u5);
					\draw (u4') -- (u5');
					
					\draw[dashed] (u5) -- (d);
					\draw[dashed] (u5') -- (d);
					
					\draw[dashed, red] (u1) -- (u2');
					\draw[dashed, red] (u3) -- (u4');
					\draw[dashed, red] (u4') -- (u5);
					
				\end{tikzpicture} 
				\caption{\sloppy In this example $A_1=\{u_1,u_2\}$ and $A_2=\{u_3,u_4,u_5\}$. Indicating the facets $f_1$ and $f_2$ only by the vertices they differ in we have $f_1=(u_1,u_2,u_3,u_4,u_5)$ and $f_2=(u_1',u_2',u_3',u_4',u_5')$. In case all the dotted red edges are present, then we can define
					$f_3=(u_1,u_2',u_3,u_4,u_5)$,  $f_4=(u_1,u_2',u_3,u_4',u_5)$, $f_5=(u_1,u_2',u_3,u_4',u_5')$ and $f_6=(u_1,u_2',u_3',u_4',u_5')$ 
					such that $\{f_1,f_3,f_4,f_5,f_6,f_2\}$ is a shelling.}
				\label{fig:shelling any adjancet vertices}
			\end{figure}
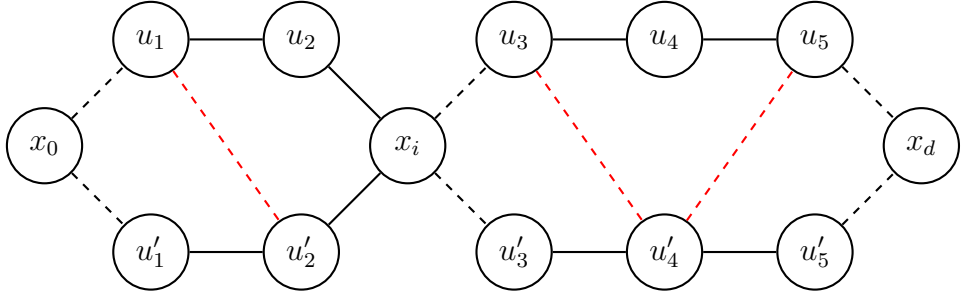
			
			\item Suppose now there are two different facets, $f_1$ and $f_2$, and suppose $\dim f_2 < \dim f_1$.

			Let $\dim f_2 = d'\leq d-1$.
			Following the structure theorem for non-pure shellable complexes provided by Lemma~\ref{lem:structure nonpure shelling} and Theorem~\ref{thm:structure nonpure shelling}, in order to produce a shelling we need to extend the $(d')$-skeleton of $f_1$ to $f_2$ by constructing a sequence of $(d')$-dimensional facets $f'_{1},\dots,f'_m$ by changing one vertex each time so that the intersection between the $j$-th facet and the preceding $(j-1)$ facets is a $(d'-1)$-dimensional simplex.
			
			If the simplices in the $(d')$-skeleton of $f_1$ and $f_2$ differ in $m \leq d'-1$ vertices, constructing such sequence is possible if we can find $\ell +2m$ edges joining the vertices in which $f_1$ and $f_2$ differ. 
			This happens with probability $p^{\ell +2m}$ and therefore following the computations done in the previous point we get, for $p = n^{-\alpha}$ and $\alpha \in [1/2, \infty)$,
			\begin{align*}
				&\sum_{m=2}^{d'-1} \binom{n}{d+1+m} n^{-\alpha (\ell + 2m)} \leq \\
				&(d-3) \binom{n}{d+3} n^{-\alpha (\ell + 4)} \sim \\
				&(d-3) \frac{n^{d+3}}{(d+3)!} n^{-\alpha (\ell + 4)}
				\xrightarrow{n\to \infty} \begin{cases}
					0, \text{ if } \alpha > \frac{d+3}{\ell +4} \\
					\infty, \text{ if } \frac{1}{2}<\alpha < \frac{d+3}{\ell +4}.
				\end{cases}
			\end{align*}

			With $p=n^{-\alpha}$, $\alpha \in [0,1/2)$, we get  
			\begin{align*}
				&\sum_{m=2}^{d'-1} \binom{n}{d+1+m} n^{-\alpha (\ell + 2m)} \leq \\
				&(d-3) \binom{n}{d+d'} n^{-\alpha (\ell + 2(d'-1))} \sim \\
				& (d-3) \frac{n^{d+d'}}{(d+d')!} n^{-\alpha (\ell + 2d' -2)}
				\xrightarrow{n\to \infty} \begin{cases}
					0, \text{ if } \alpha > \frac{d+d'}{\ell + 2d' -2} \\
					\infty, \text{ if } 0<\alpha < \frac{d+d'}{\ell + 2d' -2}.
				\end{cases}
			\end{align*}

			Again, from the fact that both inequalities $\frac{d+3}{\ell +4} \geq \frac{1}{2}$ and $\frac{d+d'}{\ell + 2d' -2} \geq \frac{1}{2}$ are true if and only if $d \geq \frac{\ell -2}{2}$, we conclude that we can construct a shelling when
			\[
			\begin{cases}
				0< \alpha < \frac{d+4}{\ell +4}, \text{ if } d\geq \frac{\ell -2}{2} \\
				0<\alpha< \frac{d+d'}{\ell + 2d' -2}, \text{ if } d< \frac{\ell -2}{2}.
			\end{cases}
			\]		
			
			\item Suppose there are $t$ facets $f_1,\dots,f_{t-1},f_t$ ordered in decreasing order with $\dim f_1=d$, then we only need to iterate the observations made in point (3).
			
			That is, at each step $j \in [1,...t-1]$ we have a shelling when
			\[
			\begin{cases}
				0< \alpha < \frac{\dim f_j +3}{\ell +4} , \text{ if } \dim f_j \geq \frac{\ell -2}{2} \\
				0<\alpha< \frac{\dim f_j +\dim f_{j+1}}{\ell + 2\dim f_{j+1} -2}, \text{ if } \dim f_{j}< \frac{\ell -2}{2}.
			\end{cases}
			\]
			
			Therefore, suppose $d=\dim f_1 < \frac{\ell -2}{2}$.
			Then every smaller facet $f_k$ will be such that $\dim f_k < \frac{\ell -2}{2}$ and we will have a shelling when 
			\[\alpha < \prod_{i=1}^{t-1} \frac{\dim f_i + \dim f_{i+1}}{\ell +2\dim f_{i+1} -2}.
			\]

			On the other hand, if $d=\dim f_1 \geq \frac{\ell -2}{2}$ let $f_k$ be the first facet in the sequence $f_1,\dots,f_t$ such that $\dim f_k < \frac{\ell -2}{2}$.
			Then we will have a shelling when 
			\[\alpha < \prod_{i=1}^{k-1} \frac{\dim f_i +3}{\ell +4} \prod_{i=k}^{t-1} \frac{\dim f_i + \dim f_{i+1}}{\ell  +2\dim f_{i+1} -2}.\]
			
		\end{enumerate}
	\end{proof}

	\begin{corollary}
		\label{cor:ETsub shellable}
		Let $G(n,n^{-\alpha})$ be an ER graph.
		Suppose the facets $g_1,\dots,g_{\tau-1},g_{\tau}$ of $\ETsub(a,b)$ are ordered in decreasing dimension.
		Then as $n \to \infty$ $\ETsub(a,b)$ is shellable asymptotically almost surely when
		\begin{itemize}
			\item $0< \alpha < \prod_{i=1}^{\tau-1} \frac{\dim g_i + \dim g_{i+1}}{(\ell -1) +2\dim g_{i+1} -2}$, if $\dim g_1  < \frac{(\ell -1) -2}{2}$,
			\item  $0<\alpha< \prod_{i=1}^{k-1} \frac{\dim g_i +3}{(\ell -1) +4} \prod_{i=k}^{\tau-1} \frac{\dim g_i + \dim g_{i+1}}{(\ell -1) +2\dim g_{i+1} -2}$, if $\dim g_i \geq \frac{(\ell -1) -2}{2}$ for $1 \leq i \leq k-1$ and $\dim g_i < \frac{(\ell -1) -2}{2}$ for $i \geq k$.
		\end{itemize}
		
	\end{corollary}
	
	It was shown in both~\cite{farmer1978cellular} and~\cite{bjorner1983lexicographically} that a shellable simplicial complex has the homotopy type of a wedge of spheres. 
	
	Therefore using Theorem~\ref{thm:ET(a,b) shellable} and Corollary~\ref{cor:ETsub shellable} we can show the following.  
	
	\begin{theorem}
		\label{thm:torsion free}
		Let $G(n,n^{-\alpha})$ be an ER graph.
		For any pair of vertices $(a,b)\in V^2$ consider the eulerian Asao-Izumihara chain complex $C_{\ast -2}(ET_{\leq \ell}(a,b), ET_{\leq \ell-1}(a,b)) \cong EMC_{\ast,\ell}(a, b)$.
		Suppose the facets $f_1,\dots,f_t$ of $ET_{\leq \ell}(a,b)$ and $g_1,\dots,g_{\tau}$ of $ET_{\leq \ell-1}(a,b)$ are ordered in decreasing dimension.
		As $n \to \infty$, in the regimes where both $\ET(a,b)$ and $\ETsub(a,b)$ are shellable, $EMH_{k,\ell}(a,b)$ is torsion free for every $k$.
		
	\end{theorem}
	
	\begin{proof}
		In the regimes where both $\ET(a,b)$ and $\ETsub(a,b)$ are shellable we can assume 
		\[\ET(a,b) \simeq \bigvee_{i=1}^{t} S_i^{n_i} \hspace{1cm}\text{and} \hspace{1cm} \ETsub(a,b) \simeq \bigvee_{j=1}^{{\tau}} S_j^{n_j}.\]
		
		So, $H_k \left(\ET(a,b), \ETsub(a,b)\right) \cong H_k \left( \vee S^{n_i}, \vee S^{n_j}\right)$, and considering the long exact sequence
		\[
		\cdots \to H_k(\vee S^{n_j}) \to H_k(\vee S^{n_i}) \to H_k \left( \vee S^{n_i}, \vee S^{n_j}\right) \to  H_{k-1}(\vee S^{n_j}) \to  \cdots
		\]
		we see that 
		\[
		H_k \left(\ET(a,b), \ETsub(a,b)\right) \cong H_k \left( \vee S^{n_i}, \vee S^{n_j}\right) \cong \begin{cases}
			\mathbb{Z}^{m_i}, \text{ if } k=n_i, \\
			\mathbb{Z}^{m_j}, \text{ if } k=n_j, \\
			0, \text{ otherwise.}
		\end{cases}
		\]
		
		Finally, from the isomorphism theorem~\ref{thm:asao-izu isomorphism thm} proved in~\cite{asao2021geometric}, we can conclude that $EMH_{k,\ell}(a,b)$ is torsion free for every $k$.
	\end{proof}

	Recall that~\cite[Theorem 4.4]{giusti2024eulerian} provides a vanishing threshold for the limiting expected rank of the $(\ell, \ell)$-eulerian magnitude homology in terms of the density parameter in the contexts of Erd\"os-R\'enyi random graphs. 
	
	\begin{theorem}[{\cite[Theorems 4.4]{giusti2024eulerian}}]
		\label{thm:vanishing threshold}
		Let $G = G(n,n^{-\alpha})$ be an Erd\"os-R\'enyi random graph. Fix $\ell$ and let $\alpha > \frac{\ell+1}{2\ell-1}$. 
		As $n \to \infty,$ $\E\left[\beta_{\ell,\ell}(n, n^{-\alpha})\right] \to 0$ asymptotically almost surely.
	\end{theorem}
	
	\begin{remark}
		\label{rem:non-vanishing (k,k) means k-facet}
		Notice that when the smallest facet of $\ET(a,b)$, $f_t$, is such that $\dim f_t  \sim \ell > \frac{\ell -2}{2}$, then $\ET(a,b)$ is shellable when 
		\[
		\alpha < \prod_{i=1}^{t-1}\left(\frac{\dim f_i +3}{\ell +4} \right) \sim 
		\prod_{i=1}^{t-1}\left(\frac{\ell +3}{\ell +4} \right) \sim 1.
		\] 
		
	\end{remark}
	
	Therefore, putting together Remark \ref{rem:non-vanishing (k,k) means k-facet} with Theorems~\ref{thm:torsion free} and~\ref{thm:vanishing threshold} we have the following.
	
	\begin{corollary}
		\label{cor:if nonvanishing then torsion free}
		Let $G(n,n^{-\alpha})$ be an Erd\H{o}s-R\'{e}nyi random graph.
		When the smallest facet $f_t$ of $\ET(a,b)$ and the smallest facet $g_{\tau}$ of $\ETsub(a,b)$ are such that $\dim f_t, \dim g_{\tau} \sim \ell$, if $EMH_{k,\ell}(G(n,n^{-\alpha}))$ is non-vanishing it is also torsion free.	
	\end{corollary}

	\section{Future directions}
	\label{sec:future directions}
	
	In this paper we investigated the regimes where an Erd\"os-R\'enyi random graph $G$ has torsion free eulerian magnitude homology groups.
	
	While the results presented have provided significant insights into the problem, several aspects remain unexplored, offering fertile ground for continued research.
	
	In this section, we propose extensions of the current work and identify open questions that could deepen the understanding of the topic.
	
	\subsection{The choice of $\ell$}
	
	The result stated in Corollary \ref{cor:if nonvanishing then torsion free} relies on the dimension of the minimal facet $f_t$ of $\ET(a,b)$ and the minimal facet $g_{\tau}$ of $\ETsub(a,b)$ being ``close enough'' to the parameter $\ell$ so that $\frac{\dim f_i +3}{\ell +4} \sim 1$ and $\frac{\dim g_j +3}{\ell +4} \sim 1$ for every other facet $f_i,g_j$. 
	
	It is thus natural to ask, how do we choose $\ell$ so that $\dim f_t \sim \ell$?
	
	First, notice that the parameter $\ell$ cannot be too big with respect to the number of vertices $n$.
	Specifically, $\ell$ cannot be of the order $n^2$.
	Indeed, suppose we pick $\ell = \frac{n(n+1)}{2}$.
	The only way we can produce a facet $f$ inducing a path of such length is if we have a path graph on $n$ vertices $V=\{1,\dots,n\}$, $(a,b)=(1,\lceil n/2 \rceil)$, and we visit vertex $n-i+1$ after vertex $i$, $i \in \{1,\dots,\lfloor n/2 \rfloor\}$, i.e. $f= (1,n,2,n-1,\dots,\lceil n/2 \rceil)$. 
	Then $\dim f= n < \frac{n(n+1)}{2}$.
	See Figure \ref{fig:path graph n vertices} for an illustration.  
	
	\begin{figure}[h]
		\centering
		\begin{tikzpicture}[node distance={20mm}, thick, main/.style = {draw, circle}, minimum size=1cm]
			
			\node[main] (1) {$1$}; 
			\node[style={draw,circle,dashed}, below of=1,node distance={1.5cm}] (a) {$a$};
			\node[main] (2) [right of=1] {$2$}; 
			\node[main] (3) [right of=2] {$3$};
			
			\node[main] (4) [right of=3] {$4$};
			\node[style={draw,circle,dashed}, below of=4,node distance={1.5cm}] (b) {$b$};
			
			\node[main] (5) [right of=4] {$5$};
			\node[main] (6) [right of=5] {$6$};
			\node[main] (7) [right of=6] {$7$};
			
			\draw (1) -- (2);
			\draw (2) -- (3);
			\draw (3) -- (4);
			\draw (4) -- (5);
			\draw (5) -- (6);
			\draw (6) -- (7);
			
			\draw[dashed] (1) -- (a);
			\draw[dashed] (4) -- (b);
			
		\end{tikzpicture} 
		\caption{In this example $(a,b)=(1,4)$ and the only facet $f$ obtained by setting $\ell =28$ is $(1,7,2,6,3,5,4)$, and $\dim f=7$.}
		\label{fig:path graph n vertices}
	\end{figure}
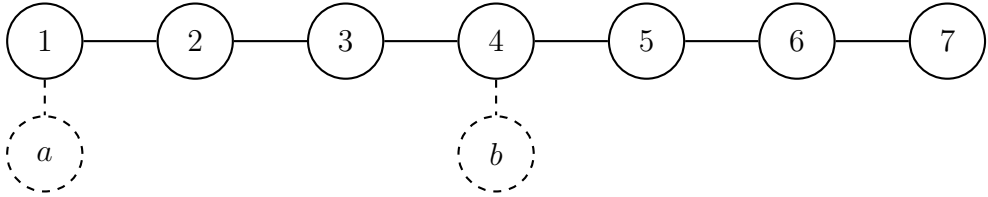
	
	We conclude that a quadratic growth rate for $\ell$ with respect to $n$ is not appropriate.
	
	On the other hand, setting $\ell = n$ we do not encounter the same problem as before.
	For example, consider the path graph in Figure \ref{fig:path graph n vertices}.
	Choosing $(a,b)=(1,4)$ and $\ell = n =7$ we find two facets $f_1=(1,2,3,6,5,4)$ and $f_2=(1,2,3,5,6,4)$.
	Both have dimension $6$ and thus $\frac{\dim f_i +3}{\ell +4} = \frac{6+3}{7+4} = \frac{9}{11} > \frac{1}{2}$.
	
	Based on this computation, along with many other examples not displayed here, we make the following conjecture.
	
	\begin{conjecture}
		Indicate the diameter of the graph $G$ by $\text{diam}(G)$.
		There exists a linear function $\varphi$ such that if $\ell \leq \varphi(\text{diam}(G))$, then $\dim f_t \sim \ell$.
	\end{conjecture}
	
	\subsection{Connection with the complex of injective words}
	
	A natural development of the work present in this paper (which we are already investigating) concerns a deterministic result about the presence of torsion in eulerian magnitude homology groups of graphs.
	It is the author's belief that this kind of result can be achieved by exploiting the strong connection between the eulerian magnitude chain complex and the complex of injective words.
	\\\\
	An \emph{injective word} over a finite alphabet $V$ is a sequence $w = v_1v_2\cdots v_t$ of distinct elements of $V$.
	Call $\text{Inj}(V)$ the set of injective words on $V$ partially ordered by inclusion, and recall that the \emph{order complex} of a poset $(P,\leq)$, denoted $\Delta(P)$, is the simplicial complex on the vertex set $P$, whose $k$-simplices are the chains $x_0 < \cdots < x_k$ of $P$.
	For example, if $P = [n]=\{1,\dots,n\}$ with the usual ordering, then $\Delta(P)=\Delta_{n-1}$ is the standard $(n-1)$-simplex.
	
	\begin{definition}
		A \emph{complex of injective words} is an order complex $\Delta(W)$ associated to a subposet $W \subset \text{Inj}(V)$.
	\end{definition}
	
	Farmer \cite{farmer1978cellular} proved that if $\#(V)=n$, then $\Delta(\text{Inj}(V))$ has the homology of a wedge of $D(n)$ copies of the $(n-1)$-sphere $S^{n-1}$, where $D(n)$ is the number of derangements (i.e. fixed point free permutations) in $\mathbb{S}_n$.
	The following result was obtained by Bj\"{o}rner and Wachs in \cite{bjorner1983lexicographically} as a strengthening of Farmer's theorem.
	
	\begin{theorem}[\cite{bjorner1983lexicographically}]
		$\Delta(\text{Inj}([n])) \simeq \bigvee_{D(n)} S^{n-1}$.
	\end{theorem}
	
	Let now the alphabet $V$ be the vertex set of a graph $G=(V,E)$.
	Let $\text{Inj}(V)$ be the set of injective words on the vertex set $V$ and denote by $\text{Inj}(V,\ell)=\{w \in \text{Inj}(V) \text{ such that } \len(w)\leq \ell \}$, the subset containing $w \in \text{Inj}(V)$ such that length of the walk $w$ in $G$ is less than $\ell$.
	Then we have a filtration 
	\begin{equation}
		\label{eq:filtration}
		\text{Inj}(V,0) \subset \text{Inj}(V,1) \subset \cdots \subset \text{Inj}(V,\ell) \subset \cdots \subset \text{Inj}(V). 
	\end{equation}

	The following equivalence easily follows from the definition of the filtration of $\text{Inj}(V)$ and the definition of the eulerian Asao-Izumihara complex $\ET(a,b)/\ETsub(a,b)$, 
	\[
	\frac{|\text{Inj}(V,\ell)|}{|\text{Inj}(V,\ell -1)|} = \bigvee_{(a,b)} \frac{|\ET(a,b)|}{|\ETsub(a,b)|},
	\]
	
	where $|\cdot|$ denotes the geometric realization.
	\\
	
	Further, the connection between the eulerian magnitude chain complex and the complex of injective words is strengthen by the following observation.  
	
	Hepworth and Roff \cite{hepworth2024bigraded} thoroughly analyzed in the context of directed graphs the \emph{magnitude-path spectral sequence (MPSS)}, a spectral sequence whose $E^1$ page is exactly standard magnitude homology, path homology \cite{grigor2019homology} can be identified with a single axis of page $E^2$, and whose target object is reachability homology \cite{hepworth2023reachability}.
	
	Reproducing the computations proposed in \cite[Section 2]{hepworth2024bigraded} using the filtration of the complex of injective words in \ref{eq:filtration}, leads to a version of the MPSS where the $E^1$ page is exactly eulerian magnitude homology.
	Since the homology of the complex of injective words, as the target object, controls the behavior of the spectral sequence, it seems reasonable to investigate the implications of this on the eulerian magnitude chain complex.
	
	\section*{Acknowledgement}
	The author is thankful to Yasuhiko Asao, Luigi Caputi and Chad Giusti for helpful conversations throughout the development of this work, and to Radmila Sazdanovic and Patrick Martin for identifying an error in an earlier version of this paper.
	
	\bibliographystyle{amsplain}
	\bibliography{bibliography}

\end{sloppypar}
\end{document}